\documentclass[a4paper]{amsart}
\usepackage{amsmath, amsfonts, amsthm, amssymb, verbatim}
\usepackage{graphicx}
\usepackage[mathcal]{euscript}
\usepackage{amstext}
\usepackage{amssymb,mathrsfs}
\usepackage{color}

\newcommand{\clg}[1]{{\mathcal{#1}}}

\newcommand{\rd}{\mathrm{d}}
\newcommand{\rn}{\mathbb{R}^{N}}
\newcommand{\rnn}{\mathbb{R}^{n}}
\newcommand{\rk}{\mathbb{R}^{k}}
\renewcommand{\H}{\mathcal{H}^{N-1}}
\newcommand{\HH}{\mathcal{H}}
\newcommand{\HN}{\mathcal{H}^{N}}
\renewcommand{\lq}{L^q(\partial\Omega)}
\renewcommand{\wp}{W^{1,p}(\Omega )}
\newcommand{\wpg}{W^{1,p}_\Gamma(\Omega)}

\renewcommand{\S}{\mathcal{S}}
\newcommand{\Se}{\mathbb{S}}
\newcommand{\X}{\mathcal{X}}
\newcommand{\R}{\mathbb{R}}

\newcommand{\spt}{\,{\rm spt}}

\newcommand{\dist}{\,{\rm dist}}

\newcommand{\dive}{\,{\rm div}}
\newcommand{\ve}{\varepsilon}

\newtheorem{te}{Theorem}[section]
\newtheorem{lem}[te]{Lemma}
\newtheorem{co}[te]{Corollary}

\theoremstyle{remark}
\newtheorem{ob}[te]{Remark}
\newtheorem{ex}[te]{Example}

\theoremstyle{definition}

\numberwithin{equation}{section}

\begin{document}

\title[Optimal boundary holes]{Optimal boundary holes for the Sobolev trace constant}

\author[L. Del Pezzo and J. Fernandez Bonder and W. Neves]
{Leandro Del Pezzo, Julian Fernandez Bonder, Wladimir Neves}

\address{Leandro M. Del Pezzo \hfill\break\indent
CONICET and Departamento  de Matem\'atica, FCEyN, Universidad de Buenos Aires,
\hfill\break\indent Pabell\'on I, Ciudad Universitaria (1428),
Buenos Aires, Argentina.}

\email{{\tt ldpezzo@dm.uba.ar}}

\address{Juli\'an Fern\'andez Bonder \hfill\break\indent
CONICET and Departamento  de Matem\'atica, FCEyN, Universidad de Buenos Aires,
\hfill\break\indent Pabell\'on I, Ciudad Universitaria (1428),
Buenos Aires, Argentina.}

\email{{\tt jfbonder@dm.uba.ar}\hfill\break\indent {\it Web page:}
{\tt http://mate.dm.uba.ar/$\sim$jfbonder}}

\address{Wladimir Neves \hfill\break\indent
Instituto de Matematica, Universidade Federal do Rio de Janeiro,
\hfill\break\indent  C.P. 68530, Cidade Universit\'aria 21945-970, Rio de Janeiro, Brazil.}

\email{{\tt wladimir@im.ufrj.br}}

\subjclass[2000]{35J66, 47R05}

\keywords{Steklov eigenvalues, p-laplace operator, shape optimization}

%
\begin{abstract}
In this paper we study the problem of minimizing the Sobolev trace Rayleigh quotient $\|u\|_{W^{1,p}(\Omega)}^p / \|u\|_{L^q(\partial\Omega)}^p$ among functions that vanish in a set contained on the boundary $\partial\Omega$ of given boundary measure.

We prove existence of extremals for this problem, and analyze some particular cases where information about the location of the optimal boundary set can be given. Moreover, we further study the shape derivative of the Sobolev trace constant under regular perturbations of the boundary set.
\end{abstract}
%
\maketitle

\section{Introduction} \label{INT}

Sobolev inequalities have proved to be a fundamental tool in order to study differential equations. Among Sobolev inequalities, one that have capture a great deal of attention in recent years is the Sobolev trace inequality that states
$$
S\left(\int_{\partial\Omega} |u|^q\, \rd\H\right)^{p/q} \le \int_{\Omega} |\nabla u|^p + |u|^p\, \rd x,
$$
for every $u\in W^{1,p}(\Omega)$ for some constant $S>0$, $1\le q\le p_*$, where  $p_*$ is the critical exponent in the Sobolev trace immersion, i.e. $p_* = p(N-1)/(N-p)$ if $1<p<N$ and $p_*=\infty$ if $p\ge N$ (the equality $q=p_*$ does not hold in the limit case $p=N$). Here $\mathcal H^s$ denotes, as usual, the $s-$dimensional Hausdorff measure, $\Omega\subset\rn$ is a smooth bounded domain (Lipschitz will be enough for most of our arguments).

In these inequalities, a fundamental role are played by the {\em optimal constants} and their associated {\em extremals}. That is, respectively, the largest possible constant $S$ in the above inequality defined as
$$
S = S_{p,q}(\Omega) := \inf_{u\in \X}\frac{\int_{\Omega} |\nabla u|^p + |u|^p\, \rd x}{\left(\int_{\partial\Omega} |u|^q\, \rd\H\right)^{p/q}}
$$
and extremals, which are functions $w\in \X $ where the above infimum is attained. Here $\X$ is the space of {\em admissible functions}, $\X:=W^{1,p}(\Omega)\setminus W^{1,p}_0(\Omega)$.

It is a well known fact that if $1<p<N$ and $1\le q\le p_*$ or $p\ge N$ and $1\le q<\infty$ then the constant $S$ is positive. For the existence of extremals, the only case which
is nontrivial is the {\em critical} one, $1<p<N$ and $q=p_*$ where
the immersion $W^{1,p}(\Omega)\subset L^{p_*}(\partial\Omega)$ is
no longer compact. (see, for instance \cite{FBR1, FBR2}).

The critical case (i.e. $1<p<N$ and $q=p_*$) was analyzed in \cite{FBR4} and \cite{FBS}. In those papers the authors show that, under very mild assumptions on the domain $\Omega$ (e.g. the existence of a boundary point of positive mean curvature) there exist extremals for $S$.

\medskip

Motivated by some problems in shape optimization for stored energies under prescribed loadings, in \cite{FBRW2} the authors study a variant of the trace inequality (see \cite{FBRW2} for further discussion on the problem): Given a set $A\subset \Omega$, minimize the {\em Rayleigh quotient} over the class of functions that vanishes on $A$, i.e.
$$
S(A):= \inf_{u\in \X_A}\frac{\int_{\Omega} |\nabla u|^p + |u|^p\, \rd x}{\left(\int_{\partial\Omega} |u|^q\, \rd\H\right)^{p/q}}
$$
where
$$
\X_A:=\{u\in \X\colon u=0 \ \mbox{a.e. on\; } A\}.
$$
In the above mentioned paper \cite{FBRW2}, existence of extremals for $S(A)$ is proved in the subcritical case $q<p_*$ (see \cite{FBS} for the critical case) and moreover the following shape optimization problem is studied: Minimize $S(A)$ among measurable sets $A\subset \Omega$ such that $\HN(A) = \alpha \HN(\Omega)$ for some fixed $0<\alpha<1$. A set $A^*$ that minimizes $S(A)$ is called an optimal set.

In \cite{FBRW2} the existence of optimal sets is established and some geometric properties of optimal sets are analyzed. Moreover, in the case $p=2$ the interior regularity of optimal sets is studied in \cite{FBRW1}. See \cite{FBRS}, where some asymptotic behavior of optimal sets are studied (see also, Section 4). Further, in \cite{FBGR} and in \cite{DP} the so-called {\em shape derivative} for $S(A)$ is computed with respect to regular deformations on the set $A$.

\medskip

One observes that, in all the above mentioned works, the sets where the test functions are forced to vanish are {\em interior} sets, i.e. $A\subset\Omega$ of positive Lebesgue measure. However, the important case of boundary sets, i.e. $\Gamma \subset \partial\Omega$ was not treated previously. Hence, the main objective of this work is to fill this gap.

\medskip

So, in this paper we study the best Sobolev trace constant from $W^{1,p}(\Omega)$ into $L^q(\partial\Omega)$ for functions that vanish on a subset $\Gamma$ of $\partial\Omega$, i.e.
\begin{equation}\label{riquelme}
S(\Gamma):= \inf_{u\in \X_\Gamma}\frac{\int_{\Omega} |\nabla u|^p + |u|^p\, \rd x}{\left(\int_{\partial\Omega} |u|^q\, \rd\H\right)^{p/q}}
\end{equation}
where
$$
\X_\Gamma:=\{u\in \X\colon u=0 \ \H-\mbox{a.e. } \Gamma\}.
$$

\medskip
Here, we consider exponents $1\le q<p_*,$ so that the immersion
$W^{1,p}(\Omega)\subset L^q(\partial\Omega)$ turns out to be
compact. Therefore, the existence of extremals for $S(\Gamma)$
follows by direct minimization.

The critical case, could be treated by the same method employed in \cite{FBS}. However, we will not do it in this article.

\bigskip

Next, we study the following optimization problem: Given $0<\alpha<1,$ we look for the value
\begin{equation}\label{aguero}
    \S(\alpha):=
    \inf \left\{
    S(\Gamma)\colon\Gamma\subset\partial\Omega,
    \H(\Gamma)=\alpha \, \H(\partial\Omega)
    \right\}.
\end{equation}

\medskip

A set $\Gamma^* \subset \partial \Omega$ is called an {\em optimal
boundary hole}, when it realizes the above infimum, i.e.
$S(\Gamma^*) = \S(\alpha)$ and $\H(\Gamma^*) = \alpha
\H(\partial\Omega)$.

One of the main issues of this paper is to show the existence and geometric properties of optimal boundary holes.

\subsection*{Organization of the paper} The rest of the paper is organized as follows. After a short section 2 were we collect some preliminary remarks, in section 3 we establish the existence of optimal boundary holes. In section 4, we analyze the simpler case where the domain $\Omega$ is a euclidean ball given a complete characterization of optimal boundary holes for this simpler geometry. In order to have a better understanding of more complex geometries, in section 5 we use a dimension reduction technique to deal with domains that are stretched in some directions. Finally, in section 6, we compute the so--called {\em shape derivative} of $S(\Gamma)$ for regular deformations of a fixed boundary hole $\Gamma$.

\section{Preliminary remarks}

In this very short section, we give some preliminary observations that will be helpful in the remaining of the paper.

First, observe that if $u$ is an extremal for $S(\Gamma)$ then $u$ turns out to be a week solution to the following Euler--Lagrange equation
\begin{equation}\label{batista}
    \begin{cases}
        -\Delta_p u + |u|^{p-2}u = 0 & \textrm{in } \Omega,
        \\[5pt]
        |\nabla u|^{p-2}\frac{\partial u}{\partial \nu}= \lambda|u|^{q-2}u & \textrm{on } \partial\Omega\setminus
        \Gamma,\\[5pt]
        u=0 & \text{on } \Gamma,
\end{cases}
\end{equation}
where $\Delta_{p} u= \textrm{div}(|\nabla u|^{p-2} \, \nabla u)$
is the usual $p-$Laplacian, $\frac{\partial}{\partial \nu}$ is the
outer unit normal derivative and $\lambda$ is a
positive constant that depends on the normalization of $u$. This is
$u\in\X_\Gamma$ and
$$
\int_{\Omega}|\nabla u|^{p-2}\nabla u\nabla\phi + |u|^{p-2}u\phi\, \rd x =\lambda \int_{\partial\Omega} |u|^{q-2}u\phi\, \rd\H,
$$
for every $\phi\in\X_\Gamma$. Observe that, if $\|u\|_{\lq}=1,$ then $\lambda=S(\Gamma).$

\medskip

As a consequence of \eqref{batista}, we have the following remarks.

\begin{ob}\label{regularidad.interior}
By the regularity results of \cite{Li2},  an extremal $u$ of $S(\Gamma)$,  verify that $u\in C^{1,\delta}_{loc}(\Omega)$ for some $0<\delta<1$.

Moreover, by \cite{L}, if $\partial\Omega\setminus \overline{\Gamma} \in C^{1,\eta}$, then the regularity up to the boundary is $u\in C^{1,\gamma}_{loc}(\overline{\Omega}\setminus\overline{\Gamma})$ for some $0<\gamma<1.$
\end{ob}

\begin{ob}
\label{RCS} If  $u $ is an extremal of
$S(\Gamma),$ then we have that $|u|$ is also an extremal of
$S(\Gamma)$. Thus, using that $|u|$ is a week solution of
\eqref{batista} and the maximum principle (see \cite{V}), we
have that $u$ has constant sign. Therefore, we can always assume that
\[
u > 0 \mbox{ in } \Omega \mbox{ and } u\ge0 \mbox{ on } \partial
\Omega.
\]
Moreover, by Hopf's Lemma (see \cite{V}) and the boundary regularity we obtain that nonnegative solutions $u$ to \eqref{batista} verify
$$
u>0 \quad \mbox{in } \overline{\Omega}\setminus\overline{\Gamma}.
$$
\end{ob}

\medskip

Finally, we need the following lemma on pointwise convergence for
Sobolev functions. We believe that this result is well-known but
we were unable to find it in the literature.

\begin{lem}\label{conv.capacidad}
Let $\{f_n\}_{n\in\mathbb{N}}\subset W^{1,p}(\Omega)$ with $1<p<N$ be such that $f_n\to 0$ as $n\to\infty$ in $W^{1,p}(\Omega)$. Then, there exists a subsequence $\{f_{n_j}\}_{j\in\mathbb{N}}\subset \{f_n\}_{n\in\mathbb{N}}$  and a set $B\subset\overline{\Omega}$ such that $cap_p(B)=0$ and
$$
f_{n_j}(x)\to 0,\quad \mbox{as } j\to\infty\qquad \mbox{for } x\in\overline{\Omega}\setminus B.
$$
\end{lem}

\begin{proof}
The lemma is a consequence of Lemma 1 and Theorem 1 in Section 4.8
of \cite{EG}. In fact, by Lemma 1 in Section 4.8 of \cite{EG}, we
have, for $\alpha>0$, the Tchebyshev--type inequality
$$
cap_p(Mf > \alpha) \le \frac{C}{\alpha^p}\|f\|_{W^{1,p}(\Omega)}^p,
$$
where $C$ is a positive constant that depends only on
$N$, $p$ and $Mf$ is the Hardy-Littlewood maximal function. So, if $f_n\to 0$ in $W^{1,p}(\Omega)$, there
exists a subsequence, $\{f_{n_j}\}_{j\in\mathbb{N}}$ such that
$$
cap_p(Mf_{n_j}>1/j) < \frac{C}{2^j}.
$$
Let us define $A_j:=\{Mf_{n_j}>1/j\}$ and let $B_m := \cup_{j=m}^\infty A_j$. Therefore,
$$
cap_p(B_m)\le \sum_{j=m}^\infty cap_p(A_j) < C \sum_{j=m}^\infty
\frac{1}{2^j}.
$$

Now, if $x\in \Omega\setminus B_m$, $Mf_{n_j}(x) < 1/j$ and by Theorem 1, section 4.8 of \cite{EG}, it follows that $|f_{n_j}(x)| < 1/j$, so $f_{n_j}\to 0$ as $j\to\infty$ in $\Omega\setminus B_m$ for all $m\in\mathbb{N}$.

Since $cap_p(B_m)\to 0$ as $m\to\infty$ the result follows.
\end{proof}

\section{The existence an optimal boundary hole} \label{OHS}

In this section, following ideas from \cite{FBRW2}, we first
prove that $S(\Gamma)$ is lower semi-continuous with respect to the
hole (Theorem \ref{sls}). Then, we prove the existence of an optimal
boundary hole.

\begin{te}\label{sls}Let $\{\Gamma_\ve
\}_{\ve>0}$ be a family of positive $\H-$measurable subsets of
$\partial\Omega$ and $\Gamma_0\subset\partial\Omega$ be a positive
$\H-$ measurable set, such that
$$
    \chi_{\Gamma_\ve}\stackrel{*}{\rightharpoonup}
    \chi_{\Gamma_0} \quad *-\textrm{weakly in }
    L^\infty(\partial\Omega),
    $$
    where $\chi_A$ is the characteristic function of
    the set $A$. Then,
    $$
    S(\Gamma_0)\le\liminf_{\ve \to 0^+}S(\Gamma_\ve).
    $$
\end{te}

\begin{proof}
    Let $\{\Gamma_n\}_{n\in\mathbb{N}}$ be a subsequence of
    $\{\Gamma_\ve\}_{\ve>0}$ such that
    $$
    \mathcal{L}
    =\liminf_{\ve\to0}S(\Gamma_\ve)=\lim_{n\to\infty}
    S(\Gamma_n).
    $$
    For each $n\in\mathbb{N}$,
    we consider $u_n \in \X_{\Gamma_n}$ to be an
    extremal of $S(\Gamma_n)$, such that
    \[
    u_n \ge 0 \quad \text{and} \quad
    \|u_n\|_{L^q(\partial\Omega)}=1.
    \]
    Therefore, the sequence
    $\{u_n\}_{n\in\mathbb{N}}$ is bounded in $W^{1,p}(\Omega)$ and
       hence there exists a
       function $u \in W^{1,p}(\Omega)$,
       such that, for a subsequence still denoted by
       $\{u_n\}_{n\in\mathbb{N}}$,
       \begin{eqnarray}
        u_n &\rightharpoonup& u, \quad \textrm{ weakly in }
        W^{1,p}(\Omega),\\
    \label{con1} u_n &\to& u, \quad
    \textrm{ strongly in } L^p(\Omega),\\
        \label{con2} u_n &\to& u, \quad \textrm{ strongly in }
        L^q(\partial\Omega).
        \end{eqnarray}
    In particular, we have that $u\ge0,$
    $\|u\|_{L^q(\partial\Omega)}=1$ and
    $$
    \|u\|_{W^{1,p}(\Omega)}\le\liminf_{n\to\infty}
    \|u_n\|_{W^{1,p}(\Omega)}.
    $$
    Moreover, for each $n \in \mathbb{N}$, $u_n= 0$ $\H-$a.e. on $\Gamma_n$.
    Thus, as
    $$
    \chi_{\Gamma_n}\stackrel{*}{\rightharpoonup}\chi_{\Gamma_0}\quad *-\mbox{weakly
    in }L^\infty(\partial\Omega)
    $$and by
    \eqref{con2}, we have
    $$
      0= \lim_{n\to\infty}\int_{\Gamma_n}u_n \, \rd \H= \int_{\Gamma_0}u \, \rd\H.
    $$
    Therefore, since $u\ge0,$ we have that $u = 0$ $\H-$a.e. on
    $\Gamma_0.$ Thus $u$ is an admissible function in the
    characterization of $S(\Gamma_0)$ and
    $$
    S(\Gamma_0)\le\|u\|_{W^{1,p}(\Omega)}^p\le\liminf_{n\to\infty}
    \|u_n\|_{\wp}^p=\mathcal{L}.
    $$
    This finishes the proof.
\end{proof}

\begin{ob} There isn't any monotonicity assumption on the family $\{\Gamma_\ve \}_{\epsilon>0}$.
\end{ob}

The continuity of $S(\Gamma)$ with respect to the topology of Theorem \ref{sls} does not hold, as is shown in the following example.

\begin{ex} We take $1<p\le N$. The case for $p> N$ is easier by the compact embedding of $W^{1,p}(\Omega)$ into continuous functions.

Let $\Omega$ be a bounded domain in $\mathbb{R}^n$ that satisfies
    the interior ball condition for all $x\in\partial\Omega.$
    Let $x_0\in\partial\Omega$
    and let $E\subset \partial\Omega$ be set of zero $\H-$measure such that $cap_p(E)>0$ and
    there exists $r>0$ such that $B(x_0,r)\cap
    E=\emptyset.$  Then, we take $\Gamma = B_{\frac{r}{2}}(x_0)
    \cap\partial\Omega$
    and $\Gamma_n= \Gamma\cup E_n$ where $E_n=\cup_{x\in
    E}B(x,\frac{1}{n})\cap\partial \Omega$ for all $n\in\mathbb{N}.$
    Observe that
    \[
    \chi_{\Gamma_{\frac{1}{n}}}\stackrel{*}{\rightharpoonup}
         \chi_{\Gamma} \quad *-\textrm{weakly in }
          L^\infty(\partial\Omega).
    \]
    Let $u_n$ be a positive normalized extremal for $S(\Gamma_n).$
    If we assume that $S(\Gamma_n)\to S(\Gamma)$ as $n\to+\infty,$
    we have that there exist $u\in W^{1,p}(\Omega)$ such that, for
    a subsequence still denote
    $\{u_n\}_{n\in\mathbb{N}},$
    $u_n\to u$ strongly in $W^{1,p}(\Omega)$ and
    $u_n\to u$ strongly in $L^q(\partial\Omega).$ Therefore $u$ is a
    positive normalized extremal for $S(\Gamma).$
    Moreover, by the Hopf's Lemma, $u_n>0$ on
    $\partial\Omega\setminus\Gamma_n$ and
    $u>0$ on $\partial\Omega\setminus\Gamma.$

    On the other hand, by Lemma \ref{conv.capacidad}, there exists a
    subsequence
    $\{u_{n_j}\}_{j\in\mathbb{N}}$ of $\{u_n\}_{n\in\mathbb{N}}$ and a
    set $B\subset\overline{\Omega}$ such that $cap_p(B)=0$ and
    $u_{n_j}(x)\to u$ as $j\to\infty$ for
    $x\in\overline{\Omega}\setminus B.$ Then, as $u_{n_j}(x)=0$ for all
    $x\in E$ and $j\in\mathbb{N},$ and $cap_p(E)>0,$ we have that $u(x)=0$ for all
    $x\in E,$ contrary to $u>0$ on $\partial\Omega\setminus \Gamma.$
\end{ex}

\bigskip

Next we prove the existence of an optimal boundary
hole. For this, we first need to show the following lemma.

\begin{lem}
\label{L2} For each $\alpha \in (0,1)$, $\S(\alpha)$ has also
the following characterization:
\begin{equation*}
    \S(\alpha) := \inf\left\{\frac{\|v\|_{\wp}^p}{\|v\|_{L^q(\partial\Omega)}^p}\colon
    v \in \X, \;    \H(\{v=0\}) \geq \alpha \, \H(\partial\Omega)
    \right\}.
\end{equation*}
\end{lem}

\begin{proof}
    Let $\alpha\in(0,1)$ and
    \[
        \tilde{\S}(\alpha) := \inf \left\{\frac{\|v\|_{\wp}^p}{\|v\|_{L^q(\partial\Omega)}^p}\colon
    v \in \X, \;  \H(\{v=0\}) \geq \alpha \, \H(\partial\Omega) \right\}.
\]
We want to prove that $\S(\alpha)=\tilde{\S}(\alpha).$ For this, we
proceed in two steps.

\medskip

\noindent {\em Step 1}. First, we show that $\tilde{\S}(\alpha)\le\S(\alpha).$

Let $\Gamma$ be a subset of $\partial
    \Omega$ such that $\H(\Gamma)= \alpha \, \H(\partial \Omega)$.
    Let $u\in\X_\Gamma$ be a nonnegative extremal for $S(\Gamma)$.
    
Observe that, $u$ is an admissible function in the characterization of
$\tilde{\S}(\alpha)$ and
$$
\tilde{\S}(\alpha) \leq
\frac{\|u\|_{\wp}^p}{\|u\|_{L^q(\Omega)}^p}=S(\Gamma).
$$
Consequently, we have that $\tilde{\S}(\alpha)\leq \S(\alpha).$

\medskip

\noindent {\em Step 2}. Now, we show that
$\S(\alpha)\le\tilde{\S}(\alpha).$

Let $\{v_n\}_{n\in\mathbb{N}}$ be a minimizing
sequence of $\tilde{\S}(\alpha)$, i.e. $v_n\in \X$,
\[
\tilde{\S}(\alpha)=\lim_{n\to\infty}\frac{\|v_n\|_{\wp}^p}{\|v_n\|_{\lq}}
\quad \mbox{and} \quad \H(\{v_n=0\})\ge\alpha\H(\partial\Omega)\quad
\forall n\in\mathbb{N}.
\]
Thus, for each $n \geq 1$, we take
$$
  \Gamma_n \subset \{v_n=0 \}
$$
such that $\Gamma_n$ is $\H-$measurable and
$\H(\Gamma_n)= \alpha \, \H(\partial \Omega)$. Thus, we
have
\[
\S(\alpha)\le S(\Gamma_n)\le
\frac{\|v_n\|_{\wp}^p}{\|v_n\|_{L^q(\Omega)}^p} \quad \forall \ n\in\mathbb{N}.
\]
then, passing to the limit in the above inequality when $n\to\infty,$ we have
$$
  \S(\alpha)\le\lim_{n \to \infty} S(\Gamma_n)=  \lim_{n \to
  \infty}\frac{\|v_n\|_{\wp}^p}{\|v_n\|_{L^q(\Omega)}^p} = \tilde{\S}(\alpha).
$$
The proof is complete.
\end{proof}

\bigskip

Now, we establish the main results of this section.
\begin{te}\label{messi1} Let $0<\alpha<1$. Then, there exist:
    \renewcommand{\labelenumi}{$($\alph{enumi}$)$}
    \begin{enumerate}
        \item A set $\Gamma_0 \subset \partial \Omega$,
            such that $\H(\Gamma_0)= \alpha \, \H(\partial \Omega)$
            and $\S(\alpha) =S(\Gamma_0)$;
        \item A function $u \in \X$ with
            $\H(\{u=0\}) \geq \alpha \, \H(\partial \Omega)$,
            such that
            $$
            \S(\alpha) =\frac{\|u\|^p_{\wp}}{\|u\|_{\lq}^p}.
            $$
    \end{enumerate}
\end{te}

\begin{proof} We divide the proof into two steps.

{\em Step 1}. First, we prove (b).

Let $\{v_n\}_{n\in\mathbb{N}}$ be a nonnegative normalized minimizing
sequence for $\S(\alpha)$, i.e. for each $n \ge 1$,
\[
  0\le v_n\in \X,\quad \|v_n\|_{L^q(\partial \Omega)} = 1, \quad
\H(\{v_n=0\}) \ge \alpha \, \H(\partial
  \Omega), \]
and
\[
  \lim_{n\to\infty} \|v_n\|^p_{W^{1,p}(\Omega)}
  = \S(\alpha).
\]
Thus the sequence $\{v_n\}_{n\in\mathbb{N}}$ is bounded in
$W^{1,p}(\Omega)$ and, therefore there exists a function $u\in\wp$ and
a subsequence still denote $\{v_n\}_{n\in\mathbb{N}}$ such that
    \begin{eqnarray}
        \label{T1con1}
        v_n &\rightharpoonup& u \quad \textrm{ weakly in }
        W^{1,p}(\Omega),\\
        \label{T1con2} v_n &\to& u \quad \textrm{ strongly in } L^p(\Omega),\\
        \label{T1con3} v_n &\to& u \quad \textrm{ strongly in }
        L^q(\partial\Omega), \\
        \label{T1con4} v_n &\to& u \quad \textrm{$\H$-a.e. in }
        (\partial\Omega).
    \end{eqnarray}
From \eqref{T1con3} and \eqref{T1con4}, we have that
$\|u\|_{L^q(\partial \Omega)}= 1$ and
$$
  \H(\{u=0\}) \geq \limsup_{n \to \infty}
  \H(\{v_n=0\}) \geq \alpha \, \H(\partial
  \Omega).
$$
Thus, $u$ is an admissible function in the definition of
$\S(\alpha)$, and therefore
$$
  \S(\alpha) \leq \|u\|^p_{W^{1,p}(\Omega)}.
$$
The reverse inequality is clear, since from \eqref{T1con1}
$$
  \|u\|^p_{W^{1,p}(\Omega)} \leq
  \lim_{n \to \infty} \|v_n\|^p_{W^{1,p}(\Omega)} = \S(\alpha).
$$

{\em Step 2}. We show that (b) implies (a).

By (b), there exists $u\in\X$ such that
$\H(\{u=0\}) \geq \alpha \, \H(\partial \Omega)$ and
$$
\S(\alpha) =\frac{\|u\|^p_{\wp}}{\|u\|^p_{\lq}}.
$$
Thus, there exists a set $\Gamma_0\subset
\{x\in\partial\Omega\colon u(x)=0\}$ $\H-$mesurable such that
$$
 \H(\Gamma_0)= \alpha \, \H(\partial \Omega).
$$
Then
we have that
\[
S(\Gamma_0)\le \frac{\|u\|^p_{\wp}}{\|u\|^p_{\lq}}=\S(\alpha),
\]
and $\H(\Gamma_0)=\alpha\H(\partial\Omega).$ Therefore
\[
\S(\alpha)=S(\Gamma_0).
\]
This finishes the proof.
\end{proof}

\medskip

In the next Theorem we make a refinement of Theorem \ref{messi1} and prove, under further regularity assumptions on $\partial\Omega$, that for any extremal $u\in\X$, it holds that $\H(\{u=0\}) = \alpha \H(\partial\Omega)$ (i.e. $\Gamma_0 = \{u=0\}$ with the notation of the above proof).

\begin{te}\label{messi2} Let $u\in \X$ be an  extremal of $\S(\alpha).$ Then, if $\Omega$ satisfies the interior ball condition, we have that
$$
  \H(\{u=0\})= \alpha \,\H(\partial\Omega).
$$
\end{te}

\begin{proof}
Let $u\in \X$ be an extremal of $\S(\alpha),$ i.e.
$\H(\{u=0\}) \geq \alpha \, \H(\partial\Omega)$ and
\[
\S(\alpha)=\frac{\|u\|_{\wp}^p}{\|u\|_{\lq}^p}.
\]
By contradiction, suppose the thesis were false, then
$$
  \H(\{u=0\}) > \alpha \, \H(\partial\Omega).
$$
Since $\clg{H}^s$ is a Borel regular measure $(0 \leq s <
\infty)$, see \cite{EG}, there exists a closed set $\Gamma_0\subset
\{x\in\partial\Omega\colon u(x)=0\}$ such that
$$
  \H(\{u=0\}) > \H(\Gamma_0)> \alpha \, \H(\partial \Omega).
$$
Consequently, it follows that
$$
  \S(\alpha) \leq S(\Gamma_0).
$$
On the other hand, the function $u$ is admissible in the
characterization of $S(\Gamma_0)$, hence
$$
  S(\Gamma_0) \leq\frac{\|u\|^p_{\wp}}{\|u\|^p_{\lq}}=\S(\alpha).
$$
Therefore, $\S(\alpha)=S(\Gamma_0)$ and so $u$ is also an extremal of
$S(\Gamma_0).$ Thus $u$ is a week solution of the following problem
\begin{equation}
    \begin{cases}
    \label{PT2}
        -\Delta_p u + |u|^{p-2}u = 0 & \textrm{in } \Omega,
        \\[5pt]
        |\nabla u|^{p-2}\frac{\partial u}{\partial \nu}=
        \lambda\, |u|^{q-2}u & \textrm{on } \partial\Omega\setminus
        \Gamma_0,\\[5pt]
        u=0 & \text{on } \Gamma_0,
    \end{cases}
\end{equation}
where $\lambda$ depends on the normalization of $u.$
Moreover, by Remark \ref{regularidad.interior},
$u \in C^{1,\gamma}_{loc}(\Omega \cup (\partial \Omega
\setminus \Gamma_0))$ for some $0<\gamma<1$ and we can assume that $u>0$
in $\Omega.$

Now, by our assumption on $\Omega$ we can apply Hopf's Lemma (cf. Remark \ref{RCS}), to get
\[
\frac{\partial u}{\partial\nu}>0 \quad  \mbox{on }
\{x\in\partial\Omega\colon u(x)=0\}\setminus\Gamma_0.
\]
That is a contradiction.
\end{proof}

\begin{co} The set function $\S$ is strictly increasing with respect to $\alpha$.
\end{co}

\begin{proof}
 It is clear that $\S(\alpha)$ is nondecreasing.
 Now, if we suppose that there exists $0<\alpha<\beta<1$,
 such that $\S(\alpha)=\S(\beta)$, then an
 extremal for $\S(\beta)$ is also an extremal for
 $\S(\alpha)$.
But, if $u$ is an extremal for $\S(\beta)$, then
$$
\H(\{u=0\})=\beta \, \H(\partial\Omega)> \alpha \, \H(\partial\Omega),
$$
which is a contradiction to Theorem \ref{messi2}. Thus,
$\S$ is strictly increasing.
\end{proof}

\bigskip
\section{Example: the unit ball} \label{RSSC}

Now, we study symmetry properties of optimal holes
in the special case where $\Omega$ is the unit ball, $\Omega=B(0,1)$. First, we
recall some of the definitions and results concerning spherical
caps. We address the reader to \cite{k, Sp}.

\medskip

\subsection*{Spherical Symmetrization.} Given a measurable set
$A\subset\rn,$ the spherical symmetrization $A^*$ of $A$ is
constructed as follows: for each positive $r,$ take $A\cap\partial B(0,r)$
and replace it by the spherical cap of the same $\H-$measure and center
$re_{N}.$ This can be done for almost all $r.$ The union of
these caps is $A^*.$ Now, the spherical symmetrization $u^*$ of a given
measurable function $u\geq0$ defined on $\Omega$ is constructed by symmetrizing the
super-level sets so that, for all $t,$ $\{u^*\geq t\}=\{u\geq
t\}^*.$ See \cite{k, Sp}.

The following theorem is proved in \cite{Sp} (see also \cite{k}).

\begin{te}[\cite{Sp}] \label{ss} Let $u\in W^{1,p}(B(0,1))$ and let $u^*$ be its
spherical symmetrization. Then $u^*\in W^{1,p}(B(0,1))$ and
\begin{equation} \label{4.20}
\begin{array}{l}
\displaystyle \int_{B(0,1)} |\nabla u^*|^p \, \rd x \leq \int_{B(0,1)} |\nabla
u|^p \, \rd x,\\[12pt]
\displaystyle  \int_{B(0,1)} | u^*|^p \, \rd x  = \int_{B(0,1)} |
u|^p \, \rd x,\\[12pt]
\displaystyle \int_{\partial B(0,1)} | u^*|^q \, \rd \H = \int_{\partial B(0,1)}
|u|^q \, \rd \H.\\
\end{array}
\end{equation}
\end{te}

In this case we can prove the following.

\begin{te}\label{sc}
    Let $\Omega=B(0,1)$ and let $0<\alpha<1.$ Then, there exists an optimal boundary hole
    which is a spherical cap. Moreover, when $p=2,$ $\Gamma$ is an
    optimal boundary hole if, and only if $\Gamma$ is a spherical cap
    (up to sets of zero $\H-$measure).
\end{te}

\begin{proof}
    Fix $\alpha\in(0,1)$, by the Theorem \ref{messi1}, there exists
    a function $u\in \X$ such that
    $\H(\{u=0\})=\alpha\H(\partial B(0,1))$ and
    \[
    \S(\alpha)= \frac{\|u\|_{W^{1,p}(B(0,1))}^p}{\|u\|_{L^q(\partial
    B(0,1))}^p}.
    \]
    Let $u^*$ be  the spherical symmetrization
    of $u.$ Then $u^*$ is an  admissible function in the definition
    of $\S(\alpha)$ and, by Theorem \ref{ss},
    \[
    \S(\alpha)\le\frac{\|u^*\|_{W^{1,p}(B(0,1))}^p}{\|u^*\|_{L^q(\partial
    B(0,1))}^p}\le \frac{\|u\|_{W^{1,p}(B(0,1))}^p}{\|u\|_{L^q(\partial
    B(0,1))}^p}=\S(\alpha).
    \]
    Therefore
    \begin{equation}
    \S(\alpha)=\frac{\|u^*\|_{W^{1,p}(B(0,1))}^p}{\|u^*\|_{L^q(\partial
    B(0,1))}^p}.
	    \label{auxiliar}
    \end{equation}
    Moreover,  $\Gamma:=\{x\in\partial B(0,1)\colon
    u^*(x)=0\}$ is a spherical cap and, since
    $\H(\{u=0\})=\alpha\H(\partial B(0,1)),$  we have that
    $\H(\Gamma)=\alpha\H(\partial B(0,1)).$ Then, using \eqref{auxiliar}, we get that
    \[
    \S(\alpha)= S(\Gamma).
    \]

    Now consider $p=2$. Let $\Gamma$ be an optimal boundary hole
    and let $u$ be an extremal of $S(\Gamma).$
    In this case, it is proved in \cite{D} that
    if equality holds in \eqref{4.20} then for each $0<r\le 1$ there
    exists a rotation $R_r$ such that
    \begin{equation}\label{rot}
        u\mid_{\partial B(0,r)} = (u^*\circ R_r)\mid_{\partial B(0,r)}.
    \end{equation}
    We can assume that the axis of symmetry $e_N$ was taken so that
    $R_1=Id$. Therefore $u$ and $u^*$ coincide on
    $\partial B(0,1)$. Then the set $\{x\in\partial B(0,1)\colon
    u(x)=0\}$ is an spherical cap and, by Theorem \ref{messi2},
    $\H(\{u=0\})=\alpha\H(\partial B(0,1)).$ \end{proof}

\section{Dimension reduction} \label{DR}

In this section, we are interested in the characterization of optimal boundary holes, when we shrink some of the dimensions of the set $\Omega$. This procedure of dimension reduction is interesting for such domains $\Omega$, where one of the directions is smaller than other ones. We begin with a fundamental case when the set $\Omega$ is given by a cartesian product, then we extend our
results for more general domains.

The ideas in this section follow closely the ones in \cite{FBMR} where the behavior of the best Sobolev trace constant for shrinking domains was analyzed and \cite{FBRS} where the interior set problem was studied.

\bigskip

\subsection{The product case} Let $\Omega_1$ and $\Omega_2$ be bounded domains respectively in
$\rnn$ and $\rk$, which are connected and have smooth boundaries.
Set $\Omega= \Omega_1 \times \Omega_2$ and for some \mbox{$0<\mu
<1$}, define
\begin{equation}
  \label{OMU}
  \Omega_\mu =\Omega_1 \times \mu \, \Omega_2
  = \{ (x, \mu y) \colon (x,y) \in \Omega \}.
\end{equation}
It is easy to see that $\partial \Omega_\mu= \overline{\Omega}_1 \times
\mu\partial\Omega_2
 \cup \partial \Omega_1 \times
\mu\overline{\Omega}_2$ and
\begin{equation}
     \H(\partial \Omega_\mu)=
  \mu^{k-1} \clg{H}^n(\Omega_1) \; \clg{H}^{k-1}(\partial \Omega_2)
  + \mu^k \clg{H}^{n-1}(\partial \Omega_1) \; \clg{H}^k(\Omega_2),
    \label{mborde}
\end{equation}
where we recall that $N=n+k$. Moreover we see that, formally,
$\Omega_1$ represents the boundary of $\Omega_\mu$ in the limiting
process. This fact will be made clear a posteriori.

\medskip
Now let $u_\mu$ be a function defined in $\Omega_\mu$. We define, for each $(x,y) \in \Omega$,
$$
  v_\mu(x,y)= u_\mu(x, \mu y).
$$
Then, $v_\mu$ is defined in $\Omega$ and enjoys the same regularity than $u_\mu$. More precisely, we have the following

\begin{lem}
\label{LUVM} If $u_\mu \in W^{1,p}(\Omega_\mu)$, then $v_\mu \in W^{1,p}(\Omega)$. Moreover,
$$
\begin{aligned}
    \H\left(\{u_\mu= 0\} \cap \partial \Omega_\mu\right)
    =&\mu^{k-1} \; \H\left(\{v_\mu= 0\} \cap
    (\Omega_1 \times \partial \Omega_2) \right)
    \\
    &+ \mu^k \;
    \H\left(\{v_\mu= 0\} \cap (\partial \Omega_1 \times
    \Omega_2)
    \right).
\end{aligned}
$$
\end{lem}

\begin{proof}
The regularity of $v_\mu$ is clear. On the other hand, since $\chi_B \equiv
\chi_A\circ T_\mu$, where
$$
  A= \{(x,\zeta) \in \Omega_\mu \,;\, u_\mu(x,\zeta)= 0\},
  \quad
  B= \{(x,y) \in \Omega \,;\, v_\mu(x,y)= 0\},
$$
and $T_\mu\colon \Omega\to\Omega_\mu$
$T_\mu(x,y)=(x,\mu y).$ We have that,
\begin{align*}
    \H(A)&=\int_{\partial\Omega_\mu}\chi_A\,\rd\H\\
    &=\iint_{\Omega_1\times\mu\partial\Omega_2}\chi_A\,\rd\mathcal{H}^{k-1}
    \rd x+
        \iint_{\partial\Omega_1\times\mu\Omega_2}\chi_A\,\rd\mathcal{H}^{n-1}
    \rd y\\
    &=\mu^{k-1}\iint_{\Omega_1\times\partial\Omega_2}\chi_B\,\rd
    \mathcal{H}^{k-1}\rd x+
        \mu^k\iint_{\partial\Omega_1\times\Omega_2}\chi_B\,\rd\mathcal{H}^{n-1}
    \rd y\\
    &=\mu^{k-1}\H\left(B\cap\left(\Omega_1\times\partial\Omega_2\right)
    \right)+\mu^k
    \H\left(B\cap\left(\partial\Omega_1\times\Omega_2\right)\right).
\end{align*}
The proof is now complete.
\end{proof}

\medskip
In the remainder of this section,
we consider subcritical exponents $1\le q<p^*,$ where $p^*$ is the critical exponent for the Sobolev
embedding $W^{1,p}(\Omega_1)\hookrightarrow L^q(\Omega_1)$, given by
$$
  p^*= \frac{pn}{n-p} \mbox{ if } 1\le p<n \mbox{ or } p^*=\infty \mbox{ if } p\ge n.
$$

\medskip

Given $\alpha,\mu \in (0,1)$, we define
\[
\S_\mu(\alpha) := \inf\left\{S(\Gamma)\colon
\Gamma\subset\partial\Omega_\mu,
\H(\Gamma)\ge\alpha\H(\partial\Omega_\mu)\right\}
\]
and
\[
\Se(\alpha) := \inf\left\{\frac{\|v\|^p_{W^{1,p}(\Omega_1)}}{\|v\|_{L^q(\Omega_1)}^p}
\colon v\in W^{1,p}(\Omega_1),\, \mathcal{H}^{n}\left(\left\{x\in\Omega_1\colon
v(x)=0\right\}\right)\ge\alpha\mathcal{H}^n(\Omega_1)\right\}.
\]
Observe that $\Se(\alpha)$ is the best Sobolev constant of the embedding
$W^{1,p}(\Omega_1)\subset L^q(\Omega_1)$ for functions that vanish
on a subset of $\Omega_1$ of a given positive measure greater than or
equal to $\alpha\HH^n(\Omega_1).$

\medskip

\begin{ob}\label{aux11}
Arguing as in section 2 (cf. with \cite{FBRW2} where the interior set case is studied), we can prove that for every $0<\alpha<1$ there exists $v_\alpha\in W^{1,p}(\Omega_1)$ such that
\[\HH^n(\{x\in\Omega_1\colon v_\alpha(x)=0\})=\alpha\HH^n(\Omega_1)
\quad \mbox{ and }\quad
\Se(\alpha)=\frac{\|v_\alpha\|^p_{W^{1,p}(\Omega_1)}}
{\|v_\alpha\|_{L^q(\Omega_1)}^p}.
\]
Moreover, $\Se(\alpha)$ is strictly increasing as a function of $\alpha.$
\end{ob}

Next, we give a characterization of the asymptotic, as $\mu\to 0^+$, behavior of $\S_\mu(\alpha)$. In fact, we see that, properly rescaled, the limit behavior is given by $\Se(\alpha).$

In order to do this, we need a couple of lemmas. The first one is easy and was proved in \cite{FBGR}.

\begin{lem}[\cite{FBGR}, Lemma 3.1]
    \label{laux0}Let $\Omega_1\subset\mathbb{R}^n$ be a domain and let
    $f_j,f\colon\Omega_1\to\mathbb{R}$ be nonnegative
    measurable functions ($j=1, 2, \dots$) such that $f_j\to f$ a.e. in
    $\Omega_1.$ Set $A_j=\{x\in\Omega_1\colon f_j(x)=0\}$ and
    $A=\{x\in\Omega_1\colon f(x)=0\}$ and suppose that
    $\HH^n(A_j)\to\HH^n(A)$ as $j\to+\infty.$ Then
    \[
    \lim_{j\to+\infty}\H(A_j\Delta A)=0.
    \]
\end{lem}

The second lemma gives the right continuity of $\Se(\alpha)$ with respect to $\alpha$.

\begin{lem}\label{laux1}
Let $1\le p<n$, $1\le q < p^*$ and $0<\alpha_0 <1$.
Then,
\[
\lim_{\alpha\to\alpha_0^+}\Se(\alpha) = \Se(\alpha_0).
\]
Moreover, if we denote by $v_\alpha$ a nonnegative extremal for $\Se(\alpha)$
normalized such that $\|v_\alpha\|_{L^q(\Omega_1)}=1,$ then there exists a
sequence $\{\alpha_j\}_{j\in\mathbb{N}}$, $\alpha_j>0$ for every $j\in\mathbb{N}$, such that $\alpha_j\to \alpha_0^+$ as $j\to+\infty$ and
\begin{equation}
    \lim_{j\to+\infty}v_{\alpha_j} = v \quad\mbox{strongly in }
W^{1,p}(\Omega_1),
    \label{confuerte}
\end{equation}
where $v$ is a nonnegative extremal for $\Se(\alpha_0).$

Lastly, if $A_j = \{x\in\Omega_1\colon v_{\alpha_j}(x)=0\}$ and
$A=\{x\in\Omega_1\colon v(x)=0\},$ we have that
\begin{equation}
    \lim_{j\to+\infty}\HH^n(A_j\Delta A)=0.
    \label{conmed}
\end{equation}
\end{lem}

\begin{proof}For this, we proceed in three steps.

\noindent {\em Step 1}. First, we prove that
$\Se(\alpha)\to\Se(\alpha_0)$ as $\alpha\searrow \alpha_0 .$

We begin by observing that, since $\Se(\cdot)$ is increasing by Remark \ref{aux11}, there exists
    \begin{equation}
    \mathcal{L}=\lim_{\alpha \to \alpha_0^+}\Se(\alpha)\quad \mbox{ and } \quad
    \mathcal{L}\ge\Se(\alpha_0).
        \label{primero}
    \end{equation}

        On the other hand, by Remark    \ref{aux11}, there exists $v_{\alpha_0}\in W^{1,p}(\Omega_1)$
    an extremal of $\Se(\alpha_0)$
    such that $\|v_{\alpha_0}\|_{L^q(\Omega_1)}=1$ and
    \[
    \HH^n(A_{\alpha_0})=\alpha_0\HH^n(\Omega_1),
    \]
    where $A_{\alpha_0} = \{x\in\Omega_1\colon v_{\alpha_0}(x) = 0\}$.

Now we choose a smooth function $\eta$ satisfying
    \[
    \begin{cases}
    \eta = 0 \mbox{ in } B(0,1), \\
    \eta = 1 \mbox { in } \mathbb{R}^n\setminus
    B(0,2),\\
    0\le\eta\le1 \mbox{ and }
    \|\nabla\eta\|_{L^\infty(\mathbb{R}^n)}\le 2 .
        \end{cases}
    \]

Take $x_0\in\Omega_1\setminus A_{\alpha_0}$ a point of density one (see definition in Chapter 1.7 of \cite{EG}) and for each $\varepsilon>0,$ set
    $\eta_\varepsilon(x)=\eta(\frac{x-x_0}{\varepsilon})$ and
    $w_\varepsilon=\eta_\varepsilon v_{\alpha_0} \in
    W^{1,p}(\Omega).$ Observe that
    \begin{equation}
        \HH^n\left(\{x\in\Omega_1\colon w_\varepsilon(x)=0\}\right)>
        \alpha_0\HH^n(\Omega_1),
        \label{med1}
    \end{equation}
     for $\varepsilon$ sufficiently small and
    \begin{equation}
        \lim_{\varepsilon\to0^+}\|w_\varepsilon\|_{L^q(\Omega_1)}=
        \|v_\alpha\|_{L^q(\Omega_1)}\quad\forall q\in[1,p^*].\label{wconv}
            \end{equation}
    Moreover
    \begin{align*}
        \|\nabla
        w_\varepsilon\|_{L^p(\Omega_1)}
        &\le\|\nabla\eta_\varepsilon
        v_{\alpha_0} + \eta_\varepsilon \nabla
        v_{\alpha_0}\|_{L^p(\Omega_1)}\\
        &\le\|\nabla\eta_\varepsilon
        v_{\alpha_0}\|_{L^p(\Omega_1)} + \|\nabla
        v_{\alpha_0}\|_{L^p(\Omega_1)}\\
        &\le
        \frac{C}{\varepsilon}
        \|v_{\alpha_0}\|_{L^p(B(x_0,2\varepsilon)\setminus
        B(x_0,\varepsilon))} + \|\nabla
        v_{\alpha_0}\|_{L^p(\Omega_1)}
    \end{align*}
    and, by H\"{o}lder's inequality,  we get that
    \begin{equation}
        \|\nabla
        w_\varepsilon\|_{L^p(\Omega_1)}\le
        C
        \|v_{\alpha_0}\|_{L^{p^*}(B(x_0,2\varepsilon)\setminus
        B(x_0,\varepsilon))} + \|\nabla
        v_{\alpha_0}\|_{L^p(\Omega_1)},
        \label{desimp}
    \end{equation}
    where $C$ is a constant independent of $\varepsilon$.
    Then, by \eqref{med1}, there exist $\delta>0$ such that
    \[
    \HH^n\left(\{x\in\Omega_1\colon w_\varepsilon(x)=0\}\right)>
        \alpha\HH^n(\Omega_1) \quad\forall 0<\alpha-\alpha_0<\delta.
    \]
    Therefore, $w_\varepsilon$ is an admissible function in the
    definition of $\Se(\alpha)$ and, using \eqref{desimp}, we
    have that
    \begin{align*}
        S(\alpha)&\le
        \frac{\|w_\varepsilon\|_{W^{1,p}(\Omega_1)}^p}
        {\|w_\varepsilon\|_{L^q(\Omega_1)}^p}\\
        &\le\frac{\left(C\|v_{\alpha_0}\|_{L^{p^*}
        (B(x_0,2\varepsilon)\setminus
        B(x_0,\varepsilon))} + \|\nabla
        v_{\alpha_0}\|_{L^p(\Omega_1)}\right)^p+
        \|w_\varepsilon\|_{L^p(\Omega_1)}^p}
        {\|w_\varepsilon\|_{L^q(\Omega_1)}^p}
    \end{align*}
    for all $\alpha>\alpha_0.$ Then, by \eqref{primero},
    \[
    \mathcal{L}\le\frac{\left(C\|v_{\alpha_0}\|_{L^{p^*}
        (B(x_0,2\varepsilon)\setminus
        B(x_0,\varepsilon))} + \|\nabla
        v_{\alpha_0}\|_{L^p(\Omega_1)}\right)^p+
        \|w_\varepsilon\|_{L^p(\Omega_1)}^p}
        {\|w_\varepsilon\|_{L^q(\Omega_1)}^p}\quad\forall\varepsilon>0.
    \]
    Lastly, taking limit as $\varepsilon\to0^+$ and using
    \eqref{wconv} and \eqref{primero}, we get that
    \[
    \mathcal{L}\le\frac{\|v_{\alpha_0}\|_{W^{1,p}(\Omega_1)}^p}
    {\|v_{\alpha_0}\|_{L^q(\Omega_1)}^p}=\Se(\alpha_0)\le\mathcal{L}.
    \]
    Then, we have that
    \begin{equation}
        \lim_{\alpha\to\alpha_0^+}\Se(\alpha)=\Se(\alpha_0),
        \label{contSS}
    \end{equation}
    as we wanted to show.

    \medskip

    \noindent {\em Step 2}. Now, we prove that \eqref{confuerte}
    holds.

    Let $v_\alpha$ be a nonnegative extremal for
    $\Se(\alpha)$ normalized such that
    $\|v_\alpha\|_{L^q(\Omega_1)}=1.$ Thus, by \eqref{contSS}, we have
    that
    \begin{equation}\label{contSS1}
        \Se(\alpha_0)=\lim_{\alpha\to\alpha_0^+}\Se(\alpha)=\lim_{\alpha\to\alpha_0^+}
    \|v_\alpha\|_{W^{1,p}(\Omega_1)}^p,
    \end{equation}
    and therefore $\{v_\alpha\}$ is bounded in $W^{1,p}(\Omega_1).$
    Then, there exists a sequence $\{\alpha_j\}$ such that
    $\alpha_j\to\alpha_0^+$ as $j\to+\infty$ and
    \begin{eqnarray}
        \label{TT1ccon1}
    v_{\alpha_j} &\rightharpoonup& v \quad \textrm{ weakly in }
        W^{1,p}(\Omega_1),\\
    \label{TT1ccon2} v_{\alpha_j} &\to& v \quad \textrm{ strongly in }
    L^p(\Omega_1),\\
    \label{TT1ccon3} v_{\alpha_j} &\to& v \quad \textrm{ strongly in }
        L^q(\Omega_1), \\
    \label{TT1ccon4} v_{\alpha_j} &\to& v \quad \textrm{$\HH^n$-a.e. in }
        (\Omega_1),
    \end{eqnarray}
    where $v\in W^{1,p}(\Omega_1).$ Since
    $\|v_{\alpha_j}\|_{L^q(\Omega_1)}=1$ for all $j\in\mathbb{N},$ using
    \eqref{TT1ccon3}, we have that $\|v\|_{L^q(\Omega_1)}=1$ and by
    \eqref{TT1ccon4} $v$ is nonnegative. By
    \eqref{contSS1}, \eqref{TT1ccon1} and \eqref{TT1ccon2}, we get that
    \begin{equation}\label{contSS3}
        \Se(\alpha_0)=\lim_{j\to+\infty}
        \|v_{\alpha_j}\|_{W^{1,p}(\Omega_1)}^p\ge\|v\|_{W^{1,p}(\Omega_1)}^p,
    \end{equation}
    and using \eqref{TT1ccon4}, we have that
    \begin{equation}
    \alpha_0\HH^n(\Omega_1)\le\liminf_{j\to+\infty}\HH^n(A_j)
    \le\limsup_{j\to+\infty}\HH^n(A_j)\le\HH^n(A),
        \label{desmed}
    \end{equation}
    where $A_j=\{x\in\Omega_1\colon v_{\alpha_j}(x)=0\}$ and
    $A=\{x\in\Omega_1\colon v(x)=0\}.$ Then, $v$ is an admissible
    function in the definition of $\Se(\alpha_0),$ and using
    \eqref{contSS3}, we get that
    \[
    \Se(\alpha_0)\le\|v\|_{W^{1,p}(\Omega_1)}^p\le\Se(\alpha_0).
    \]
    Therefore $v$ is an extremal for $\Se(\alpha_0)$ and, by
    \eqref{contSS1}, we have
    \begin{equation}
        \lim_{j\to+\infty}\|v_{\alpha_j}\|_{W^{1,p}(\Omega_1)}=
    \|v\|_{W^{1,p}(\Omega_1)}.
        \label{casiter1}
    \end{equation}
    Moreover, using \eqref{TT1ccon1} and \eqref{casiter1}, we can
    conclude that
    \[
    \lim_{j\to+\infty}v_{\alpha_j} = v \quad\mbox{strongly in }
    W^{1,p}(\Omega_1).
    \]

    \medskip

    \noindent {\em Step 3}. Lastly, we prove that \eqref{conmed}
    holds.

    First, we prove that $\HH^n(A)=\alpha_0\HH^n(\Omega_1).$ On the
    contrary, suppose that $\HH^n(A)>\alpha_0\HH^n(\Omega_1),$ then
    there exists $j_0$ such that
    $\HH^n(A)>\alpha_j\HH^n(\Omega_1)$ for all $j\ge
    j_0$ and therefore
    \[
    \Se(\alpha_0)=\|v\|_{W^{1,p}(\Omega_1)}^p>\Se(\alpha_j)>\Se(\alpha_0)
    \]
    and we obtain a contradiction. Thus
    $\HH^n(A)=\alpha_0\HH^n(\Omega_1)$ and by \eqref{desmed}
    \[
    \lim_{j\to+\infty}\HH^n(A_j)=\HH^n(A).
    \]
    Then, by \eqref{TT1ccon4} and Lemma \ref{laux0}, we have that
    \[
    \lim_{j\to+\infty}\HH^n(A_j\Delta A)=0.
    \]
    This finishes the proof.
\end{proof}

We arrive now at the main result of this section.

\begin{te}
Let $0<\alpha,\mu<1,$ $1\le p<n,$ and $1\le q< p^*,$ then
\[
\lim_{\mu\to0^+}\frac{\S_\mu(\alpha)}{\mu^{\frac{k(q-p)+p}{q}}}= \frac{\HH^k(\Omega_2)} {\HH^{k-1}(\partial\Omega_2)^{\frac{p}{q}}} \Se(\alpha).
\]
\end{te}
\begin{proof}
    We begin by proving
    \[
    \limsup_{\mu\to0^+}\frac{\S_\mu(\alpha)}{\mu^{\frac{k(q-p)+p}{q}}}\le
    \frac{\HH^k(\Omega_2)}{\HH^{k-1}(\partial\Omega_2)^{\frac{p}{q}}}\Se(\alpha).
    \]

Let
    \[
    \alpha_\mu=\alpha\left(1+\mu\frac{\HH^{n-1}
    (\partial\Omega_1)\HH^{k}(\Omega_2)}
    {\HH^n(\Omega_1)\HH^{k-1}(\partial\Omega_2)}\right).
    \]
and take $v\in W^{1,p}(\Omega_1)$ such that
    \[
    \mathcal{H}^{n}\left(A\right)\ge\alpha_\mu\mathcal{H}^n(\Omega_1),
    \]
    where
    \[
    A=\left\{x\in\Omega_1\colon v(x)=0\right\}.
    \]

Then, if we take $u(x,y)=v(x)$ for all $(x,y)\in\Omega_\mu,$ we have that
    \begin{align*}
        \H\left(\{w=0\}\cap\partial\Omega_\mu\right)&\ge
        \H\left(\{w=0\}\cap
        \left(\overline{\Omega}_1\times\mu\partial\Omega_2\right)
        \right)\\
        &\ge \H\left(A\times\mu\partial\Omega_2\right)\\
        &=\mu^{k-1}\HH^n(A)\HH^{k-1}(\partial\Omega_2)\\
        &\ge\mu^{k-1}\alpha_\mu
        \HH^n(\Omega_1)\HH^{k-1}(\partial\Omega_2)\\
        &=\alpha\H(\partial\Omega_\mu).
    \end{align*}
    Therefore, $u$ is an admissible function in the characterization
    of $\S_\mu(\alpha)$ (see Lemma \ref{L2}), then
    \begin{align*}
        \S_\mu(\alpha)&\le\frac{\iint_{\Omega_\mu}|\nabla
        w|^p+|w|^p\,\rd x\rd
        y}{\left(\int_{\partial\Omega_\mu}|w|^q\,\rd\H
        \right)^{\frac{p}{q}}}\\
        &=\frac{\mu^k\HH^k(\Omega_2)\int_{\Omega_1}|\nabla
        v|^p+|v|^p\,\rd
        x}{\left(\mu^{k-1}\HH^{k-1}(\partial\Omega_2)\int_{\Omega_1}|v|^q
        \rd x +
        \mu^k\HH^k(\Omega_2)\int_{\partial\Omega_1}|v|^q\,\rd
        \HH^{n-1}\right)^{\frac{p}{q}}}\\
        &\le
        \mu^{\frac{k(q-p)+p}{q}}\frac{\HH^k(\Omega_2)}{\HH^{k-1}
        (\partial\Omega_2)^{\frac{p}{q}}}\frac{\int_{\Omega_1}|\nabla
        v|^p+|v|^p\,\rd x}{\left(\int_{\Omega_1}|v|^q \rd
        x\right)^{\frac{p}{q}}}.
    \end{align*}
    Thus, taking infimum over all $v\in W^{1,p}(\Omega_1)$ such
    that
    \[
    \HH^n\left(\left\{x\in\Omega_1\colon v(x)=0
    \right\}\right)\ge\alpha_\mu\HH^n(\Omega_1),
    \]
    we get that
    \[
    \frac{\S_\mu(\alpha)}{\mu^{\frac{k(q-p)+p}{q}}}\le \frac{\HH^k(\Omega_2)}{\HH^{k-1}
        (\partial\Omega_2)^{\frac{p}{q}}}
    \Se(\alpha_\mu).
    \]
    Therefore, using Lemma \ref{laux1},
    \begin{equation}
        \limsup_{\mu\to 0^+}
        \frac{\S_\mu(\alpha)}{\mu^{\frac{k(q-p)+p}{q}}}\le \frac{\HH^k(\Omega_2)}{\HH^{k-1}
        (\partial\Omega_2)^{\frac{p}{q}}}
    \Se(\alpha).
        \label{limsup}
    \end{equation}

    On the other hand, for each $\mu$ there exist there exists an
    extremal $u_\mu\in W^{1,p}(\Omega_\mu)$ of $\S_\mu(\alpha)$ such that
    \begin{equation}
        \iint_{\Omega_1\times\partial\Omega_2}|v_\mu|^q \rd x\rd
        \HH^{k-1} + \mu\iint_{\partial\Omega_1\times\Omega_2}
        |v_\mu|^q \rd \HH^{n-1}\rd y = 1,
        \label{enorm}
    \end{equation}
    where $v_\mu(x,y)=u_\mu(x,\mu y).$

    Then,
    \begin{align*}
        \S_\mu(\alpha)&=\frac{\iint_{\Omega_\mu}|\nabla
        u_\mu|^p+|u_\mu|^p \,\rd x\rd
        y}{\left(\int_{\partial\Omega_\mu} |u_\mu|^q \, \rd
        \H\right)^{p/q}}\\
        &=\frac{\iint_{\Omega}\left(|(\nabla_x v_\mu,
        \mu^{-1}\nabla_y v_\mu)|^p+|v_\mu|^p\right)\mu^k \,\rd x\rd
        y}{\left(\mu^{k-1}
        \iint_{\Omega_1\times\partial\Omega_2}|v_\mu|^q \rd x\rd
        \HH^{k-1} + \mu^k\iint_{\partial\Omega_1\times\Omega_2}
        |v_\mu|^q \rd \HH^{n-1}\rd y\right)^{\frac{p}{q}}}\\
        &=\mu^{\frac{k(q-p)+p}{q}}\left(
        \iint_{\Omega}|(\nabla_x v_\mu,
        \mu^{-1}\nabla_y v_\mu)|^p+|v_\mu|^p\,\rd x\rd
        y\right).
    \end{align*}
    Thus,
    \begin{equation}
        \frac{\S_\mu(\alpha)}{\mu^{\frac{k(q-p)+p}{q}}}=
    \iint_{\Omega}|(\nabla_x v_\mu,
        \mu^{-1}\nabla_y v_\mu)|^p+|v_\mu|^p\,\rd x\rd
        y \quad \forall\mu\in(0,1).
        \label{pdes}
    \end{equation}
    Let $\{\mu_j\}_{j\in\mathbb{N}}$ be a sequence such
    that $\mu_j\to 0^+$ as $j\to\infty$ and
    \[
         \liminf_{\mu\to 0^+}\frac{\S_\mu(\alpha)}{\mu^{\frac{k(q-p)+p}{q}}}
     =\lim_{j\to+\infty}
     \frac{\S_{\mu_j}(\alpha)}{\mu_j^{\frac{k(q-p)+p}{q}}}.
    \]

To simplify the notation, we write $v_j$ instead of $v_{\mu_j}$ for all $j\in\mathbb{N}.$

Then, by \eqref{limsup}, we have that $\{v_j\}_{j\in\mathbb{N}}$ is bounded in $W^{1,p}(\Omega).$ Therefore, there exists a function 
$v\in W^{1,p}(\Omega)$ and a subsequence of $\{v_j\}_{j\in\mathbb{N}}$ (still denoted by $\{v_j\}_{j\in\mathbb{N}}$) such that
    \begin{eqnarray}
        \label{T1con11}
        v_j &\rightharpoonup& v \quad \textrm{ weakly in }
        W^{1,p}(\Omega),\\
        \label{T1con22} v_j &\to& v \quad \textrm{ strongly in } L^p(\Omega),\\
        \label{T1con33} v_j &\to& v \quad \textrm{ strongly in }
        L^q(\partial\Omega).
    \end{eqnarray}
   Observe that, by
    \eqref{T1con33}, we have that
     \begin{eqnarray}
        \label{T1con44} v_j &\to& v \quad \textrm{ strongly in }
    L^q(\partial\Omega_1\times\Omega_2),\\
        \label{T1con55} v_j &\to& v \quad \textrm{ strongly in }
        L^q(\Omega_1\times\partial\Omega_2),
    \end{eqnarray}
    and, using \eqref{enorm}, \eqref{T1con44} and \eqref{T1con55}, we
    get
    \[
    \iint_{\Omega_1\times\partial\Omega_2}|v|^q\, \rd x\rd\HH^{k-1}=1,
    \]
    from where we conclude that $v\not\equiv0.$

    Now, using again \eqref{limsup} and \eqref{pdes}, we have that
    there exists a constant $C$ such that
    \[
    \iint_{\Omega} |\mu_j^{-1}\nabla_y v_j|^p\,\rd x\rd y\le C
    \quad\forall j\in\mathbb{N},
    \]
    then $\{\mu_j^{-1} \nabla_y v_j\}_{j\in\mathbb{N}}$ is bounded in
    $L^p(\Omega)$
    and
     \[
    \iint_{\Omega} |\nabla_y v_j|^p\,\rd x\rd y\le C\mu_j^p
    \to0 \mbox{ as } j\to\infty.
    \]
    Therefore $v$ does not depend on $y,$ i.e. $v=v(x)$ and
    \begin{equation}
       1=\iint_{\Omega_1\times\partial\Omega_2}|v|^q\,
       \rd x\rd\HH^{k-1}=\HH^{k-1}(\partial\Omega_2)\int_{\Omega_1}
       |v|^q \rd x.
        \label{eaux2}
    \end{equation}

    On the other hand, using that $\{\mu_j^{-1} \nabla_y
    v_j\}_{j\in\mathbb{N}}$ is bounded in $L^p(\Omega),$ there exist
    $w\in L^p(\Omega)$ such that
    \[
    \mu_j^{-1}\nabla_y v_j \rightharpoonup w \quad \textrm{ weakly in }
        L^{p}(\Omega).
    \]
    Then
    \begin{align*}
     \liminf_{\mu\to 0^+}\frac{\S_\mu(\alpha)}{\mu^{\frac{k(q-p)+p}{q}}}
     &=\lim_{j\to+\infty}
     \frac{\S_{\mu_j}(\alpha)}{\mu_j^{\frac{k(q-p)+p}{q}}}\\
     &=\lim_{j\to+\infty}\iint_{\Omega}|(\nabla_x v_j,
        \mu_j^{-1}\nabla_y v_j)|^p+|v_j|^p\,\rd x\rd
        y\\
        &\ge \iint_{\Omega}|(\nabla_x v,w)|^p+|v|^p\,\rd x\rd
        y\\
        &\ge \HH^{k}(\Omega_2)\|v\|_{W^{1,p}(\Omega_1)}^p,
    \end{align*}
    and, by \eqref{eaux2}, we get
    \begin{equation}
        \liminf_{\mu\to 0^+}\frac{\S_\mu(\alpha)}{\mu^{\frac{k(q-p)+p}{q}}}
        \ge \frac{\HH^{k}(\Omega_2)}
        {\HH^{k-1}(\partial\Omega_2)^{\frac{p}{q}}}\frac{
        \|v\|_{W^{1,p}(\Omega_1)}^p}{\|v\|_{L^q(\Omega_1)}^{\frac{p}{q}}}.
        \label{eaux3}
    \end{equation}

Lastly, by \eqref{mborde},  Lemma \ref{LUVM} and since $u_{\mu_j}$ is an extremal for $\S_{\mu_j}(\alpha)$ for all $j\in\mathbb{N}$, we have that
\begin{align*}
\alpha\HH^n(\Omega_1)\HH^{k-1}(\partial\Omega_2)\le&
\H(\{v_j=0\}\cap(\Omega_1\times\partial\Omega_2))\\
&+\mu_j \H(\{v_j=0\}\cap(\partial\Omega_1\times\Omega_2))
\end{align*}
for all $j\in\mathbb{N}$.
Then, using \eqref{T1con55},  we get that
    \begin{align*}
    \alpha\HH^n(\Omega_1)\HH^{k-1}(\partial\Omega_2)
    &\le\limsup_{j\to+\infty}\H(\{v_j=0\}\cap(\Omega_1\times\partial\Omega_2))
    \\
    &\le\H(\{v=0\}\cap(\Omega_1\times\partial\Omega_2))\\
    &=\H\left( (\{v=0\}\cap\Omega_1)\times\partial\Omega_2
    \right)\\
    &=\HH^n((\{v=0\}\cap\Omega_1)\HH^{k-1}(\partial\Omega_2).
    \end{align*}
    Thus,
    \[
    \alpha\HH^n(\Omega_1)\le\HH^n(\{v=0\}\cap\Omega_1),
    \]
    and $v$ is an admissible function in the characterization of
    $\Se(\alpha).$ Then, using \eqref{limsup} and \eqref{eaux3}, we
    have that
    \[
    \frac{\HH^{k}(\Omega_2)}
        {\HH^{k-1}(\partial\Omega_2)^{\frac{p}{q}}}\Se(\alpha)\le
        \liminf_{\mu\to 0^+}\frac{\S_\mu(\alpha)}{\mu^{\frac{k(q-p)+p}{q}}}
        \le\limsup_{\mu\to
        0^+}\frac{\S_\mu(\alpha)}{\mu^{\frac{k(q-p)+p}{q}}}\le
        \frac{\HH^{k}(\Omega_2)}
        {\HH^{k-1}(\partial\Omega_2)^{\frac{p}{q}}}\Se(\alpha).
    \]
    The proof is now complete.
\end{proof}

\bigskip

\subsection{The case $n=1$}

When the limit problem is one-dimensional we can give a more precise description of the situation. So in this subsection we consider the case $\Omega_1 = (a,b)\subset \mathbb{R}$, an interval.

In \cite{FBRS} the following Theorem regarding the limit problem for $n=1$ is proved
\begin{te}[\cite{FBRS}, Theorem 1.2]\label{optimalhole}
The optimal limit constant $\Se(\alpha)$ is attained only for
a hole $A^* = (a, a+ \alpha (b-a))$ or $A^* = (b - \alpha (b-a),
b)$, that is the best hole is an interval concentrated on one
side of the interval $(a,b)$. Moreover, the optimal limit constant
is given by
$$
\Se(\alpha) = \frac{(2\pi)^p (p-1)}
{(2\alpha(b-a)p\sin{(\frac{\pi}{p})})^p} + 1.
$$
\end{te}

As a consequence of this Theorem, we have the following Corollary on the approximate shape and location of optimal boundary holes

\begin{co}
For $\mu$ small enough the best boundary hole $\Gamma_\mu$ for the domain
$\Omega_\mu = (a,b)\times\mu\Omega_2$ with measure $\H(\Gamma_\mu) = \alpha
\H(\partial\Omega_\mu)$ looks like $\Gamma_\mu \simeq (a, a+ \alpha (b-a)) \times \partial\mu
\Omega_2$ or like $\Gamma_\mu \simeq (b - \alpha (b-a), b) \times \partial\mu \Omega_2$.

\end{co}

\subsection{General geometries} We finish this section by observing that, once the product case is studied, the extension of our results to more general domains $\Omega$ in $\rn$ than a product is done by a standard procedure. Cf. with \cite{FBMR, FBRS}.

So, in this case we let $\Omega_\mu=\{(x,\mu y)\colon (x,y)\in\Omega\}.$

We have the following
\begin{te}
Let $\Omega$ be a bounded and Lipschitz domain in $\rn.$ Let $\Omega_x$ be the $x-$section of $\Omega$ and $P(\Omega)$ be the projection of $\Omega$ onto de $x$ variable, i.e.
$$
\Omega_x := \{y\in\R^k\colon (x,y)\in\Omega\} \qquad \mbox{and}\qquad
P(\Omega):=\{x\in\R^n\colon \Omega_x\neq\emptyset\}.
$$

Then, if we call $\rho(x) = \HH^k(\Omega_x)$ and $\beta(x) = \HH^{k-1}(\partial\Omega_x)$ we have that
    \[
    \lim_{\mu\to0^+}\frac{S_\mu(\alpha)}{\mu^{\frac{k(q-p)+p}{q}}}
    =\Se(\alpha,\rho,\beta),
    \]
    where
    \[
    \Se(\alpha,\rho,\beta)=\inf\left\{\frac{\int_{P(\Omega)}(|\nabla v|^p
    + |v|^p)\,\rho(x)\rd x}{\left(\int_{P(\Omega)}|v|^q\beta(x)\rd
    x\right)^{\frac{p}{q}}}\colon v\in A(\alpha)\right\}
    \]
    with
    \[
    \mathcal{A}(\alpha)=\left\{v\in W^{1,p}(P(\Omega),\rho)\colon \HH^n(\{x\in
    P(\Omega)\colon v(x)=0\})\ge\alpha\HH^n(P(\Omega)\right\}.
    \]
    Here $W^{1,p}(P(\Omega),\rho)$ is the weighted Sobolev space,
    \[
    W^{1,p}(P(\Omega),\rho)=\left\{v\colon
    P(\Omega)\to\mathbb{R}\colon \int_{P(\Omega)}(|\nabla v|^p
    +|v|^p)\rho(x)\,\rd x<+\infty\right\}.
    \]
\end{te}

    \begin{proof} Once the product case is studied, the extension to general geometries is analog to Theorem 1.1 in \cite{FBMR}. See also Theorem 1.3 in \cite{FBRS}. We omit the details.
    \end{proof}
    \bigskip
\section{Shape derivative} \label{SD}

In this section, we are interested in the computation of the derivative of the set function $S(\cdot)$ with respect to regular deformations of the set. The formula obtained in this way could the be used in the (numerical) computation of optimal boundary holes. This approach have been used with relevant success in similar problems. See \cite{BD, FBGR, HP, O} and references therein.

Since the domain of $S(\cdot)$ are sets contained at the boundary $\partial\Omega$ which is a manifold of codimension one, we must take deformations of sets, which stays in $\partial \Omega$.

\medskip

We begin describing  the kind of variations we are going to
consider. Let $V\colon \rn \to \rn$ be a Lipschitz field such that,
$V \cdot \nu =0$ on $\partial \Omega$, where $\nu$ is the outer unit normal vector to $\partial\Omega$, and
$$
  \spt( V)\subset \Omega_\delta := \{ x \in \rn \colon \dist(x,\partial \Omega) < \delta \}
$$
for some $\delta>0$ small, where $\spt(V)$ is the support of $V$.

Now, we consider the flow associated to the field $V$. Let
$\Phi\colon [0,\infty) \times \rn \to \rn$, satisfying

$$
  \frac{d}{dt} \Phi_t(x)= V\big(\Phi_t(x)\big),
  \quad \Phi_0(x)=  x,
$$
where $\Phi_t(\cdot) \equiv \Phi(t,\cdot)$.

It is not difficult to see that, for each $t$ fixed, $\Phi_t$ is a
diffeomorphism. Indeed, by construction of the flow, $\Phi_t$ is invertible
with inverse given by $\Phi_{-t}$.
In \cite{HP}, the following asymptotic formulas were proved
$$
  \begin{aligned}
  \Phi_t(x)&= x + t \, V(x) + o(t),\\
  D\Phi_t(x)&= Id + t \, DV(x) + o(t),\\
   D\Phi_t(x)^{-1}&= Id - t \, DV(x) + o(t),\\
  J\Phi_t(x)&= 1 + t \, \dive V(x) + o(t),\\
  J_\tau \Phi_t(x)&= 1 + t \, \dive_\tau V(x) + o(t),
\end{aligned}
$$
for all $x\in\rn,$ where $J\Phi_t$ is the Jacobian of the flow and $\dive_\tau$
denotes the tangential component of the divergence operator.

\medskip

So, given
$\Gamma\subset\partial\Omega,$ we are allowed to
define
\begin{equation}
\label{DG}
     \Gamma_t:= \Phi_t(\Gamma) \subset \partial \Omega,
\end{equation}
and
\begin{equation}
    s(t):=S(\Gamma_t).
    \label{st}
\end{equation}
Observe that $s(0)=S(\Gamma).$

\begin{ob}
By construction, the flow preserves the topology of the initial
domain. Therefore, if $\Gamma$ is a connected set, then $\Gamma_t$
will be also connected. In fact, this is one of the characteristic
of the {\em shape derivative}, opposite, for instance, to the {\em topological
derivative}, see \cite{Allaire, Barbarosie, Feijoo, Iguernane}, etc.
\end{ob}

\medskip

Our first result of this section shows that, $s(t)$ is continuous with respect to $t$ at $t=0.$

\begin{te}\label{cont}
    With the previous notation,
    \[
    \lim_{t\to 0^+} s(t) =S(\Gamma).
    \]
\end{te}

\begin{proof}
Let $u\in\X_\Gamma$ and we consider $v=u\circ\Phi_t^{-1}\in     \X_{\Gamma_t}.$ By the change of  variables formula, we have
    \[
    \int_\Omega |v|^p \, \rd x =\int_\Omega |u|^p\, \rd x
        +t\int_\Omega |u|^p \,\rd x + o(t),
    \]
    and
    \[
    \int_\Omega |\nabla v|^p \, \rd x = \int_\Omega |\nabla u|^p \, \rd x +t
    \int_\Omega \left(|\nabla u|^p \dive V -p|\nabla u|^{p-2}\langle
    \nabla u, DV^T\nabla u^T\rangle\right) \rd x + o(t).
    \]
    Then,
    \[
    \int_\Omega|\nabla v|^p+|v|^p \, \rd x =
    \int_\Omega |\nabla u|^p+|u|^p \, \rd x + t R(u) + o(t),
    \]
    where
    \[
    R(u) =\int_\Omega \left(|u|^p+|\nabla u|^p\right) \dive V \rd x-
    p\int_\Omega|\nabla u|^{p-2}\langle \nabla u, DV^T\nabla u^T\rangle
    \rd x.
    \]

On the other hand, by the change of variables formula on manifolds, see \cite{HP}, we obtain
    \[
    \int_{\partial\Omega} |v|^q \, \rd \H=
    \int_{\partial\Omega} |u|^q \, \rd \H +t
    \int_{\partial\Omega} |u|^q\dive_\tau V \, \rd \H +o(t).
    \]

    Then,
\begin{equation}\label{des1}
    \begin{aligned}
        s(t)&\le \frac{\int_\Omega |\nabla v|^p+|v|^p \,
        \rd x}{\left(\int_{\partial\Omega} |v|^q \, \rd
        \H\right)^{\frac{p}{q}}}\\
        &=\frac{\int_\Omega |\nabla u|^p+|u|^p \, \rd x + t
        R(u) + o(t)}{\left(\int_{\partial\Omega} |u|^q\,\rd\H +
        t\int_{\partial\Omega} |u|^q\dive_\tau V \,
        \rd \H +o(t)\right)^{\frac{p}{q}}},
    \end{aligned}
\end{equation}
    and therefore
    \[
    \limsup_{t\to0^+}s(t)\le\frac{\int_\Omega |\nabla
    u|^p+|u|^p \,\rd x}{\left(\int_{\partial\Omega} |u|^q \, \rd
        \H\right)^{\frac{p}{q}}} \quad \forall u\in\X_\Gamma.
    \]
    Then
    \begin{equation}
          \limsup_{t\to0^+}s(t)\le S(\Gamma).
        \label{aux1}
    \end{equation}

    Now, let $\{t_n\}_{n\in\mathbb{N}}$ such that $t_n \to
    0^+$ as $n\to \infty$ and
    \begin{equation}
        \liminf_{t\to0^+}s(t)= \lim_{n\to\infty}s(t_n).
        \label{aux2}
    \end{equation}
        For each $n\in\mathbb{N},$ let $v_n$ be an positive normalized
    extremal of $s(t_n),$ i.e. $v_n\in \X_{\Gamma_{t_n}}$, $v_n>0$ in $\Omega,$
    $\|v_n\|_{\lq}=1$ and
    \begin{equation}
        s(t_n)=\int_\Omega |\nabla v_n|^p+|v_n|^p \,
        \rd x.
        \label{aux3}
    \end{equation}
    Using \eqref{aux1} and \eqref{aux2}, we have that
    $\{v_n\}_{n\in\mathbb{N}}$ is bounded in $\wp$ and therefore there exists
    $u\in\wp$ and some subsequence of $\{v_n\}_{n\in\mathbb{N}}$ (still denote
    $\{v_n\}_{n\in\mathbb{N}}$) such that
        \begin{eqnarray}
        \label{con00}v_n &\rightharpoonup& u, \quad \textrm{ weakly in }
        W^{1,p}(\Omega),\\
    \label{con11} v_n &\to& u, \quad
    \textrm{ strongly in } L^p(\Omega),\\
    \label{con22} v_n &\to& u, \quad \textrm{ strongly in }
        L^q(\partial\Omega).
        \end{eqnarray}
    Then, $u\ge0$ and $\|u\|_{\lq}=1$ and
    \[
    \liminf_{t\to0^+}s(t)\ge \int_\Omega |\nabla u|^p+|u|^p
    \, \rd x.
    \]
    On the other hand, since $\Phi_{-t}\to Id$ in the $C^1$
    topology when $t\to 0$ and using \eqref{con22}, we have
    \[
    \int_{\partial\Omega}u\chi_{\Gamma}\,\rd\H=0
    \]
    and therefore $u\in\X_{\Gamma}.$ Then, using \eqref{aux1}
    \[
    S(\Gamma)\le \int_\Omega |\nabla u|^p+|u|^p\, \rd x
    \le \liminf_{t\to0^+}s(t)
    \le\limsup_{t\to0^+}s(t)\le S(\Gamma).
    \]
    Thus,
    \[
    \lim_{t\to0^+}s(t)= S(\Gamma).
    \]
The proof is now completed.
\end{proof}

\begin{ob} \label{stcon}
    Observe that, in the above prove, we really have that $v_n\to
    u$ strongly in
    $\wp$ when $n\to\infty$ because $\|v_n\|_{\wp}\to\|u\|_{\wp}$
    when $n\to\infty$
    and by \eqref{con00}.
\end{ob}

Now we arrive at the main result of this section.

\begin{te}
If $\Gamma\subset\partial\Omega$ is a positive $\H-$measurable subset,
we have that $s(t)$ is differentiable at $t=0$ and
\begin{equation}\label{dev}
\frac{\rd s}{\rd t}(0)=-\frac{p}{q}S(\Gamma)\int_{\partial\Omega}|u|^q\dive_\tau V\,\rd\H+R(u),
\end{equation}
where
    \[
    R(u) =\int_\Omega \left(|u|^p+|\nabla u|^p\right) \dive V \rd x-
    p\int_\Omega|\nabla u|^{p-2}\langle \nabla u, DV^T\nabla u^T\rangle
    \rd x
    \]
and $u$ is an extremal of $S(\Gamma).$
\end{te}

\begin{proof}
    Let $u$ be a positive normalized extremal of $S(\Gamma).$ Then,
    using \eqref{des1}, we have that
    \[
    s(t)\le\frac{S(\Gamma)+ t R(u) + o(t)}{\left(1+
    t\int_{\partial\Omega}|u|^q\dive_{\tau}V\,\rd\H +
    o(t)\right)^{\frac{p}{q}}}.
    \]
    Thus, for all $t>0$
    \begin{align*}
    \frac{s(t)-S(\Gamma)}{t}\le& \frac{S(\Gamma)}{t}\frac{1-\left(1+
    t\int_{\partial\Omega}|u|^q\dive_{\tau}V\,\rd\H +
    o(t)\right)^{\frac{p}{q}}}{\left(1+
    t\int_{\partial\Omega}|u|^q\dive_{\tau}V\,\rd\H +
    o(t)\right)^{\frac{p}{q}}}\\
    &+\frac{R(u) + o(1)}{\left(1+
    t\int_{\partial\Omega}|u|^q\dive_{\tau}V\,\rd\H +
    o(t)\right)^{\frac{p}{q}}},
\end{align*}
Therefore
\begin{equation}
    \limsup_{t\to 0^+}\frac{s(t)-S(\Gamma)}{t}\le
    -\frac{p}{q}S(\Gamma)\int_{\partial\Omega}|u|^q\dive_\tau V \,
    \rd\H+ R(u).
    \label{auxx1}
\end{equation}

On other hand, let $\{t_n\}_{n\in\mathbb{N}}$ be a positive sequence
such that  $t_n\to0^+$ when $n\to\infty,$ and
\[
\liminf_{t\to 0^+}\frac{s(t)-S(\Gamma)}{t}=\lim_{n\to\infty}
    \frac{s(t_n)-S(\Gamma)}{t_n}.
\]
Observe that, by Lemma \ref{cont}, we have that $s(t_n)\to
S(\Gamma).$ We can now proceed analogously to the proof of Lemma \ref{cont},
 and we find a subsequence of $\{t_n\}_{n\in\mathbb{N}}$ (still denote
 $\{t_n\}_{n\in\mathbb{N}}$) such that
\[
v_n\to u \quad\mbox{strongly in } \wp,
\]
where $v_n$ is an positive normalized
extremal of $s(t_n)$ for all $n\in\mathbb{N}$ and $u$ is an positive normalized
extremal of $S(\Gamma),$ see also Remark \ref{stcon}.

Thus, taking $u_n=v_n\circ \Phi_{t_n}\in\wpg,$ we get
\[
S(\Gamma)\le
\frac{s(t_n)-t_nR(v_n)+o(t_n)}{\left(1-t_n\int_{\partial\Omega}|v_n|^q\dive_{\tau}
V\,\rd\H+o(t_n)\right)^{\frac{p}{q}}}.
\]
Then
\begin{align*}
    \frac{s(t_n)-S(\Gamma)}{t_n}&\ge
    \frac{s(t_n)}{t_n}\frac{\left(1-
    t_n\int_{\partial\Omega}|c_n|^q\dive_{\tau}V\,\rd\H +
    o(t_n)\right)^{\frac{p}{q}}-1}{\left(1-
    t_n\int_{\partial\Omega}|v_n|^q\dive_{\tau}V\,\rd\H +
    o(t)\right)^{\frac{p}{q}}}\\
    &+\frac{R(v_n) + o(1)}{\left(1-
    t_n\int_{\partial\Omega}|v_n|^q\dive_{\tau}V\,\rd\H +
    o(t_n)\right)^{\frac{p}{q}}}.
\end{align*}
Therefore
\begin{equation}
    \begin{aligned}
    \liminf_{t\to 0^+}\frac{s(t)-S(\Gamma)}{t}&=\lim_{n\to\infty}
    \frac{s(t_n)-S(\Gamma)}{t_n}\\
    &\ge
    -\frac{p}{q}S(\Gamma)\int_{\partial\Omega}|u|^q\dive_\tau V \,
    \rd\H+ R(u).
\end{aligned}
    \label{auxx2}
\end{equation}

Thus, by \eqref{auxx1} and \eqref{auxx2}, we have that $s(t)$ is
differentiable at $t=0$ and \eqref{dev} holds.
\end{proof}

\begin{ob}
One observes that, we do not need in our approach the derivative
of the eigenfunctions.
\end{ob}

\begin{ob}
It would be desirable to obtain a simplification of Formula
\eqref{dev}. In many problems (cf. \cite{FBGR, HP, O}, etc) this
can be done by using, in an appropriate way, the equation
satisfied by $u$. In our case, the obstruction we have encountered
in order to do that, is the lack of regularity of $u$ at the
boundary. A similar problem was found in \cite{BD} where the
authors attempt to overcome this difficulty by working on a subset
$\Omega_\delta\subset \Omega$ and then passing to the limit
(however, the results are not completely satisfactory). In our
case, since we cannot control the normal derivative of $u$ in
$\Omega_\delta$, this approach does not seems to be feasible.
\end{ob}

\section*{Acknowledgements}
This work was partially supported by  project PROSUL (CNPq-CONICET) nro: 490329/2008-0.

J. Fern\'andez Bonder and Leandro Del Pezzo were also partially supported by Universidad de Buenos Aires under grant X078, by ANPCyT PICT No. 2006-290 and CONICET (Argentina) PIP 5478/1438.

Wladimir Neves was also partially supported by FAPERJ though the grant E-26 / 111.564/2008 entitled Analysis, Geometry and Applications and by Pronex-FAPERJ through the grant E-26/ 110.560/2010 entitled Nonlinear Partial Differential Equations.

\bibliographystyle{amsplain}
\bibliography{biblio}

\end{document}